\RequirePackage{fix-cm}
\documentclass[12pt]{article}
\usepackage{authblk}
\usepackage{hyperref}
\usepackage[cmintegrals,cmbraces]{newtxmath}
\usepackage[left=2.4cm,right=1cm]{geometry}

\usepackage{amsthm}
\usepackage{lmodern}
\usepackage{libertinus}
\usepackage{latexsym}
\usepackage{tikz-cd}
\usepackage{graphicx} 
\usetikzlibrary{cd}
\usepackage{amscd}
\usepackage[all]{xy}
\usepackage{breqn}
\usepackage{tcolorbox}
\usepackage{biblatex}
\usepackage{footmisc}
\usepackage{comment}
\usepackage{thmtools}

\declaretheorem[name=Theorem,numberwithin=section]{theorem}
\declaretheorem[name=Proposition,sibling=theorem]{proposition}
\declaretheorem[name=Lemma,sibling=theorem]{lemma}
\declaretheorem[name=Corollary,sibling=theorem]{corollary}
\declaretheorem[name=Example,numberwithin=section]{example}
\declaretheorem[name=Remark,sibling=example]{remark}

\newcommand{\lra}{\leftrightarrows}
\newcommand{\ra}{\rightarrow}

\newcommand{\da}{{\downarrow}}
\newcommand{\up}{{\uparrow}}

\newcommand{\mb}[1]{\mbox{#1}}
\newcommand{\ca}[1]{\mathcal{#1}}
\newcommand{\bd}[1]{\mathbf{#1}}
\newcommand{\mf}{\mathsf}

\newcommand{\mi}{\mathit}

\newcommand{\se}{\subseteq}
\newcommand{\sm}{\setminus}

\newcommand{\we}{\wedge}
\newcommand{\ve}{\vee}
\newcommand{\bwe}{\bigwedge}
\newcommand{\bve}{\bigvee}
\newcommand{\bca}{\bigcap}
\newcommand{\bcu}{\bigcup}
\newcommand{\opp}[1]{#1^{op}}

\newcommand{\Sc}{\mf{S}_{\cl}}
\newcommand{\So}{\mf{S}_{\op}}
\newcommand{\Ss}{\mf{S}}
\newcommand{\Se}{\mf{S}_{\ca{E}}}

\newcommand{\op}{\mathfrak{o}}
\newcommand{\cl}{\mathfrak{c}}
\newcommand{\bl}{\mathfrak{b}}

\newcommand{\rjc}{\mathfrak{rjc}}
\newcommand{\rb}{\mathfrak{rb}}
\newcommand{\rf}{\mathfrak{rf}}
\newcommand{\Fi}{\mathsf{Filt}}
\newcommand{\fso}{\Fi_{\mathcal{SO}}}
\newcommand{\fcp}{\Fi_{\mathcal{CP}}}
\newcommand{\fr}{\Fi_{\mathcal{R}}}
\newcommand{\fe}{\Fi_{\mathcal{E}}}
\newcommand{\fse}{\Fi_{\mathcal{SE}}}

\newcommand{\Om}{\Omega}
\newcommand{\pt}{\mathsf{pt}}
\newcommand{\rpt}{\mathsf{pt}_R}
\newcommand{\maxpt}{\mathsf{maxpt}}

\newcommand{\sll}{\mf{S}(L)}

\newcommand{\scl}{\mf{S}_{\cl}(L)}
\newcommand{\clo}[1]{\mb{cl}_{#1}}

\title{Raney extensions of frames: topological aspects}
\author{Anna Laura Suarez \thanks{Department of Mathematics Tullio Levi-Civita, University of Padova, 35121, Italy. email: \href{mailto:annalaurasuarez993@gmail.com}{annalaurasuarez993@gmail.com}}}
\date{}

\date{}
\begin{document}

\maketitle

\begin{abstract}

   We explore a pointfree approach to spaces which extends the category of $T_0$ spaces. We build on the work in \cite{suarez24}, and consider the category $\bd{Raney}$ of \emph{Raney extensions}, pairs $(L,C)$ where $C$ is a coframe, $L\se C$ is a frame which meet-generates it, and the inclusion $L\se C$ preserves the frame operations as well as the strongly exact meets. We show that there is a dual adjunction $\Om_R:\bd{Top}\lra \opp{\bd{Raney}}:\pt_R$, whose fixpoints are all the $T_0$ spaces. The functor $\Om_R$ maps a space $X$ to the pair $(\Om(X),\ca{U}(X))$, where $\Om(X)$ are its opens and $\ca{U}(X)$ its saturated sets. We show that for a frame $L$ the spectra of the largest and the smallest Raney extensions over it are, respectively, the classical spectrum $\pt(L)$ and the $T_D$ spectrum $\pt_D(L)$.

    We study separation axioms in this setting. We define a Raney extension $(L,C)$ to be $T_1$ if $C$ is Boolean, motivated by the fact a space $X$ is $T_1$ if and only if $\ca{U}(X)=\ca{P}(X)$. We show that a frame is subfit if and only if it admits a $T_1$ Raney extension. We show that a subfit frame is scattered if and only if it admits a unique Raney extension. Raney extensions admit variations of the \emph{density} and \emph{compactness} properties of a canonical extension. We characterize sobriety of a space $X$, as well as the $T_D$ and $T_1$ properties, purely algebraically in terms of density or compactness of $(\Om(X),\ca{U}(X))$. We use this to define sobriety for general Raney extensions, as well as the $T_D$ property. We show that all Raney extensions admit a sober coreflection, and that in the category of Raney extensions with \emph{exact} maps, they also admit a $T_D$ reflection.
    
   We explore the close connection between exactness and the $T_D$ axiom. We show that the dual adjunction between frames and spaces restricts to a dual adjunction between the category of $T_D$ spaces and the category of $\bd{Frm}_{\ca{E}}$ of frames and exact maps. We introduce the notion of \emph{exact} sublocale, a sublocale whose surjection is exact. We use the results on exactness to show that exact sublocales $\Se(L)$ form a subcolocale of $\Ss(L)$. 
    
\end{abstract}

\tableofcontents

\section{Introduction}

In this work, we introduce an extension of the classical dual adjunction between frames and spaces at the core of pointfree topology, which captures all $T_0$ spaces. A duality for $T_0$ spaces already exists, it is \emph{Raney duality}, as illustrated in \cite{bezhanishvili20}. In Raney duality, rather than mapping a space $X$ to the frame $\Om(X)$ of its open sets, we map it to the embedding $\Om(X)\se \ca{U}(X)$ of its open sets into the lattice saturated\footnote{\emph{Saturated sets} are intersections of open sets.} sets. The limitation of Raney duality is that, on the algebraic side, our objects are all of the form $(\Om(X),\ca{U}(X))$ for some space $X$, and this means that this category does not generalize $T_0$ spaces in the way that frames generalize sober spaces. In order to gain a more pointfree perspective, we consider as objects of our category pairs $(L,C)$ where $C$ is a coframe and $L\se C$ is a frame which meet-generates $C$ and such that the embedding preserves the frame operations together with strongly exact meets\footnote{\emph{Strongly exact meets} are the pointfree version of those intersections of open sets which are open. Because a meet of a collection $\{U_i:i\in I\}$ of opens in general is calculated as the interior of $\bca_i U_i$, these are exactly the meets that are preserved by the embedding $\Om(X)\se \ca{U}(X)$.}. These objects are called \emph{Raney extensions}, and they form a category called $\bd{Raney}$ whose morphisms are coframe maps which restrict to frame maps between the generating frames. Raney extensions were introduced in \cite{suarez24}, where the more algebraic aspects of these structures were studied. In this paper, we build on that work, and analyze instead the dual adjunction $\Om_R:\bd{Top}\lra\bd{Raney}^{op}:\rpt$ between Raney extensions and spaces, exploring the more topological meaning of Raney extensions. Raney extensions prove to be a useful tool to both describing spaces pointfreely and tackle questions in classical pointfree topology.

\medskip

\medskip

For Raney extensions, we have drawn from the theory of canonical extensions of distributive lattices, and in fact in \cite{suarez24} Raney extensions are introduced as a generalization of this construction. For Raney extensions, too, we have versions of the density and compactness properties of the canonical extension. We will characterize sober, $T_D$, and $T_1$ spaces in terms of density and compactness of their Raney extensions, thus giving a purely algebraic characterization of these axioms. 
We use these facts to define the notion of \emph{sober} and of $T_D$ Raney extension. We show that any Raney extension admits a sober coreflection, and that under certain conditions Raney extensions admit $T_D$ and $T_1$ reflections. 

\medskip

The study of separation axioms in pointfree topology has been quite active, and recently the results on the matter have been published in a book, see \cite{picado21}. In pointfree topology, the $T_1$ axiom has several different translations; weakest is \emph{subfitness}, see for example Chapter V of \cite{picadopultr2012frames} or Chapter II of \cite{picado21}. In point-set topology, $T_1$ spaces are characterized by all their subspaces being saturated, and this means that, for a space $X$, being $T_1$ amounts to the embedding of the opens into the saturated sets is $\Om(X)\se \ca{P}(X)$. This leads us to defining a Raney extension $(L,C)$ to be $T_1$ if and only if $C$ is Boolean. We prove that a frame admits a $T_1$ Raney extension if and only if it is subfit, giving a precise sense in which subfitness is the weakest possible frame version of the $T_1$ axiom. 

\medskip

Another important separation axiom in pointfree topology is the $T_D$ axiom. This was first introduced in \cite{Aull62}, and it is stronger than $T_0$ and weaker than $T_1$. For the importance of this axiom in pointfree topology, see \cite{banaschewskitd}. The axiom $T_D$ is a mirror image of sobriety in the following sense:
\begin{itemize}
    \item A space $X$ is sober if and only if there can be no nontrivial subspace inclusion $i:X\se Y$ such that $\Om(i)$ is an isomorphism;
    \item A space $X$ is $T_D$ if and only if there can be no nontrivial subspace inclusion $i:Y\se X$ such that $\Om(i)$ is an isomorphism.
\end{itemize}

In \cite{banaschewskitd}, the $\pt_D(L)$ spectrum of a frame is introduced, an alternative to the classical spectrum which is always a $T_D$ space. We prove that, for a frame $L$, the classical sober spectrum $\pt(L)$ is the spectrum of the largest Raney extension on $L$, whereas the $T_D$ spectrum $\pt_D(L)$ is the spectrum of the smallest one. Furthermore, for every Raney extension $(L,C)$ for its spectrum $\pt_R(L,C)$ we always have subspace inclusions 
\[
\pt_D(L)\se \pt_R(L,C)\se \pt(L).
\]

\medskip

Another axiom which has been studied quite extensively is \emph{scatteredness}. Scatteredness for a frame $L$ is defined in \cite{plewe00} and \cite{plewe02} as the property that $\mf{S}(L)$ is Boolean. In \cite{ball16}, the authors characterize the frames for which $\mf{S}_{\cl}(L)=\mf{S}(L)$ as those subfit frames such that they are scattered. Here, we show that subfit frames which are scattered coincide with those subfit frames with unique Raney extensions. 

\medskip

Another prominent structure in pointfree topology is the collection $\mf{S}_{\cl}(L)$ of joins of closed sublocales. See \cite{ball18}, \cite{joinsofclosed19}, \cite{moshier20}, \cite{ball19}, \cite{ball16}. In \cite{suarez24}, we have shown that a frame map $f:L\to M$ lifts to a map $\Sc(L)\to \Sc(M)$ if and only if it is \emph{exact}. We introduce the notion of \emph{exact} sublocale, a sublocale such that the corresponding frame surjection is exact. It turns out that for a frame $L$, the ordered collection $\Se(L)$ of exact sublocales form a subcolocale of $\sll$, the collection of all sublocales of $L$. We also show how exactness relates to the $T_D$ axiom, and show that the classical adjunction between spaces and frames restricts to an adjunction between $T_D$ spaces and frames with exact maps.

\section{Background}

\subsection{Frames and spaces}
A \emph{frame} is a complete lattice $L$ satisfying the distributivity law $a\wedge \bigvee B = \bve \{a\wedge b:b\in B\}$ for all $a\in L$ and $B\subseteq L$. Frames form a category $\bd{Frm}$, whose morphisms are functions preserving arbitrary joins (including the bottom element $0$) and finite meets (including the top element $1$). Morphisms of frames preserve all joins, and as such they have right adjoints. For a frame map $f:L\ra M$, we will denote as $f_*$ its right adjoint. Similarly, \emph{coframes} are the complete lattices where the dual distributive law holds, and morphisms between them are the maps preserving arbitrary meets and finite joins. Any coframe morphism $f:C\to D$ has a left adjoint, which we denote as $f^*$. 

\medskip 

Given a topological space $X$, its lattice of open sets $\Om(X)$ is always a frame, and this assignment is the object part of a functor $\Omega:{\bd{Top}\ra \bd{Frm}^{op}}$ assigning to each space its frame of opens. The correspondence between frames and topological spaces at the core of pointfree topology is an adjunction $\Om:\bd{Top}\lra \opp{\bd{Frm}}:\pt$ with $\Om\dashv \pt$. There is more than one equivalent way in the literature of defining the spectrum of a frame. Here, we define $\pt(L)$ as the set of all completely prime filters\footnote{A filter $F\se L$ is \emph{completely prime} if $\bve_i x_i\in F$ implies $x_i\in F$ for some $i\in I$. Completely prime filters are the completely prime elements of $\Fi(L)$.} of a frame $L$, topologized by defining the open sets as those of the form $\varphi_L(a)=\{P\in \pt(L):a\in P\}$. A frame $L$ is \emph{spatial} when for $a,b\in L$ such that $a\nleq b$ there is some completely prime filter containing $a$ and omitting $b$. A space $X$ is \emph{sober} when every irreducible closed set is the closure $\overline{\{x\}}$ of a unique point $x$.

\begin{theorem}\label{classicalsoberspatial}
    There is an adjunction $\Om:\bd{Top}\lra \opp{\bd{Frm}}:\pt$ with $\Om\dashv \pt$. This adjunction is idempotent, and it maximally restricts to a dual equivalence between sober spaces and spatial frames.
\end{theorem}

\subsection{Sublocales}

Sublocales are the pointfree counterparts of subspaces. Because subspace inclusions in the category $\bd{Top}$ of topological spaces are the regular monomorphisms, sublocales are defined as the regular monomorphisms in the category $\bd{Loc}$ of locales. Even though we will work in the category of frames, we follow Picado and Pultr in \cite{picadopultr2012frames} in defining a \emph{sublocale} of a frame $L$ to be a subset $S\se L$ such that: 
\begin{enumerate}
    \item It is closed under all meets;
    \item Whenever $s\in S$ and $x\in L$ we have $x\ra s\in S$.
\end{enumerate}

These requirements are equivalent to stating that $S\se L$ is a regular monomorphism in the category $\bd{Loc}$ of locales. The family $\mathsf{S}(L)$ of all sublocales of $L$ ordered by inclusion is a coframe. Meets in $\mf{S}(L)$ are set-theoretical intersections. This means that they form a closure system on the subsets of $L$. For a subset $X\se L$, we denote as $\ca{S}(X)$ the smallest sublocale containing $X$. In the following, $\ca{M}(-)$ denotes closure under meets.
\begin{lemma}\label{l:small-sl}
    For a frame $L$ and for $X\se L$, we have $\ca{S}(X)=\ca{M}(\{a\ra x:a\in L,x\in X\})$.
\end{lemma}
The top element is $L$ and the bottom element is $\{1\}$. Because $\sll$ is a coframe, there is a \emph{difference} operator on it, dual to Heyting implication, defined for sublocales $S$ and $T$ as $S{\sm}T=\bca \{U\in \sll:S\se T\cup U\}$. For a sublocale $S$, we denote the element $L{\sm}S$ as $S^*$, and we call it the \emph{supplement} of $S$. For each $a\in L$, there are an \emph{open sublocale} and a \emph{closed sublocale} associated with it. These are, respectively, $\mathfrak{o}(a)=\{a\ra b: b\in L\}$ and $\mathfrak{c}(a)=\uparrow a$. We will need a few facts about open and closed sublocales, which we gather here.
\begin{proposition}\label{manyfacts}
For every frame $L$ and $a,b,a_i\in L$ we have
\begin{enumerate}
    \item $\op(1)=L$ and $\op(0)=\{1\}$;
    \item $\cl(1)=\{1\}$ and $\cl(0)=L$;
    \item $\bve_i \op(a_i)=\op(\bve_i a_i)$ and $\op(a)\cap \op(b)=\op(a\we b)$;
    \item $\bca_i \cl(a_i)=\cl(\bwe_i a_i)$ and $\cl(a)\ve \cl(b)=\cl(a\we b)$;
    \item The elements $\op(a)$ and $\cl(a)$ are complements of each other in $\mf{S}(L)$: we have $\op(a)\cap \cl(a)=$ and $\op(a)\ve \cl(a)=L$;
    \item $\cl(a)\se \op(b)$ if and only if $a\ve b=1$, and $\op(a)\se \cl(b)$ if and only if $a\we b=0$.
\end{enumerate}
\end{proposition}
Every sublocale can be written as an intersection of sublocales of the form $\op(x)\ve \cl(y)$. Additionally, every sublocale $S\se L$ has a \emph{closure} $\mi{cl}(S)=\bca \{\cl(x):S\se \cl(x)\}$, and this is $\up \bwe S$. Every sublocale $S$ also has a \emph{fitting}, defined as $\mi{fit}(S)=\bca\{\op(x):S\se \op(x)\}$. This is a closure operator, and it is studied in \cite{clementino18}. For a coframe $C$, we say that an element $c\in C$ is \emph{linear} if $\bve_i (x_i \we c)=\bve_i x_i \we c$ for any collection $x_i\in C$.

\begin{lemma}\label{l:linear}
Complemented elements of a coframe are linear. In particular, in $\sll$ open and closed sublocales are linear.
\end{lemma}

 Also particularly important are \emph{Boolean sublocales}. For an element $a\in L$ the sublocale $\{x\to a:x\in L\}$, denoted as $\bl(a)$, is the smallest sublocale containing $a$. A sublocale is a Boolean algebra if and only if it is of this form for some $a\in L$. An element $p\in L$ is \emph{prime} when $x\we y\leq p$ implies either $x\leq p$ or $y\leq p$, for all $x,y\in L$. Elements of the form $\bl(p)$ are also called \emph{two-element sublocales}, as for $p$ prime we have $\bl(p)=\{1,p\}$.

\begin{lemma}\label{l:prime}
For a frame $L$, the following hold for each prime $p\in L$ and all elements $x,y\in L$.
\begin{itemize}
    \item $x\to p=1$ if $x\leq p$, and $x\to p=p$ if $x\nleq p$.
    \item $\bl(p)\se \op(x)$ if and only if $x\nleq p$.
    \item The element $\bl(p)$ is completely join-prime in $\mf{S}(L)$.
    \item The prime elements of $\mf{S}(L)$ are the sublocales of the form $\bl(p)$ for some prime $p\in L$.
\end{itemize}
    
\end{lemma}
\subsection{\texorpdfstring{T\textsubscript{D}}{TD} duality}

 A topological space $X$ is said to be $T_D$ if for every point $x\in X$ there are opens $U$ and $V$ such that $U{\sm}V=\{x\}$. For a frame $L$ we say that a prime $p\in L$ is \emph{covered} if whenever $\bwe_i x_i=p$ for some family $x_i\in L$ then $x_i=p$ for some $i\in I$. In \cite{banaschewskitd} the \emph{$T_D$ spectrum} of a frame $L$ is defined as the collection of covered primes of a frame, with the subspace topology inherited from the prime spectrum of $L$. This space is denoted as $\pt_D(L)$. This turns out to always be a $T_D$ space. A frame morphism $f:L\to M$ is a \emph{D-morphism} if for every covered prime $p\in L$ the prime $f_*(p)$ is covered. We call $\bd{Frm}_D$ the category of frames and D-morphisms. There is a dual adjunction $\Om:\bd{Top}\lra \bd{Frm}_D:\pt_D$, where the fixpoints on the space side are the $T_D$ spaces, and on the frame side these are the \emph{D-spatial} frames, which can be characterized as those frame such that all their elements are the meet of the covered primes above them. We will use the following two results.

 \begin{proposition}\label{allpxarecov}(\cite{banaschewskitd}, Proposition 2.3.2)
    A space $X$ is $T_D$ if and only if all elements of the form $X{\sm}\overline{\{x\}}$ are covered primes in $\Om(X)$.
\end{proposition}

Furthermore, in \cite{arrieta21} the notion of \emph{D-sublocale} is introduced. This is a sublocale $S\se L$ such that the corresponding surjection is in $\bd{Frm}_D$. We have the following result.
\begin{theorem}\label{t:d-subloc}
    For a frame $L$, the D-sublocales form a subcolocale $\mf{S}_D(L)\se \sll$. We also have a subcolocale inclusion $\Sc(L)\se \mf{S}_D(L)$.
\end{theorem}

\subsection{Exact and strongly exact meets}

For a topological space $X$, we have the \emph{specialization preorder} $\leq$, defined on its points as $x\leq y$ whenever $x\in U$ implies $y\in U$ for all open sets $U\se X$. For a space $X$, we denote as $\ca{U}(X)$ the lattice of its upsets (upper-closed sets) under the specialization preorder. 
\begin{proposition}\label{upset}
For a topological space $X$, a subset is an upset in the specialization preorder if and only if it is saturated.
\end{proposition}

A space is $T_0$ if and only if the specialization preorder is an order. A space is $T_1$ if and only if all its subsets are saturated, that is, $\ca{U}(X)=\ca{P}(X)$. A weak pointfree analogue of the $T_1$ axiom is subfitness. A frame is \emph{subfit} if whenever $x,y\in L$ are such that $x\nleq y$, there is some $u\in L$ such that $x\ve u=1$ and $y\ve u\neq 1$. In general, if a space is $T_1$ then its frame of opens is subfit, the converse does not in general hold. Recall that a meet $\bwe_i x_i$ is \emph{strongly exact} if for all $y\in L$ we have that $x_i\ra y=y$ implies $(\bwe_i x_i)\ra y=y$. A meet $\bwe_i x_i$ of a frame $L$ is \emph{exact} if for every $a\in L$ we have $(\bwe_i x_i)\ve a=\bwe_i (x_i\ve a)$. We have the following.

\begin{proposition}\label{p:pres-se}
    For every space $X$, strongly exact meets in $\Om(X)$ are open sets. This means that the embedding $\Om(X)\se \ca{U}(X)$ preserves strongly exact meets.
\end{proposition}

\begin{theorem}[see \cite{ball14}, Theorem 5.2.3]\label{t:ball}
    A $T_0$ space is $T_D$ if and only if for every exact meet $\bwe_i U_i$ in $\Om(X)$ this equals $\bca_i U_i$. This is equivalent to the embedding $\Om(X)\se \ca{U}(X)$ preserving exact meets.
\end{theorem}

A filter of a frame is \emph{strongly exact} if it is closed under strongly exact meets. We call $\fse(L)$ the ordered collection of strongly exact filters. This is a frame where meets are computed as intersections, and additionally it is a sublocale of $\mf{Filt}(L)$.  A filter is \emph{exact} if it is closed under exact meets. Exact and strongly exact filters are studied in \cite{moshier20}. There, it is also shown that the exact filters form a frame, and in particular the frame $\fe(L)$ of exact filters is a sublocale of $\fse(L)$. We have the following, proven in \cite{moshier20}.

\begin{theorem}\label{eandse}
    We have an isomorphism of coframes
    \begin{gather*}
        \mi{fitt}:\opp{\fse(L)}\cong \mf{S}_{\op}(L),\\
        F\mapsto \bca \{\op(f):f\in F\}.
    \end{gather*}
    We also have an isomorphism of frames
    \begin{gather*}
         \mb{cl}:\fe(L)\cong \mf{S}_{\cl}(L),\\
         F\mapsto \bve \{\cl(f):f\in F\}.
    \end{gather*}
\end{theorem}

We will refer to several important concrete collections of filters. Since the collection $\mathsf{Filt}(L)$ is a frame, there is a Heyting operation $\ra$ on it. In the following, whenever we write $F\ra G$ for two filters $F$ and $G$, it will be understood that we are referring to this operation. Notice that for a frame $L$ and for $a,b\in L$ we have
$\up a\ra \up b=\{x\in L:b\leq x\ve a\}$. In \cite{jakl24}, we have the following characterization of exact filters.

\begin{proposition}[\cite{jakl24}, Proposition 5.5]\label{charexact}
A filter is exact if and only if it is the intersection of filters of the form $\up a\ra \up b$ for some $a,b\in L$.
\end{proposition}

\begin{lemma}[\cite{suarez24}, Lemma 4.10]\label{principalmin}
    For a frame $L$, the sublocale $\fe(L)\se \Fi(L)$ is the smallest one containing all principal filters.
\end{lemma}
 Recall from \cite{suarez24} that a frame map $f:L\to M$ is \emph{exact} if, whenever $\bwe_i x_i\in L$ is an exact meet, we have that $\bwe_i f(x_i)$ is exact and $\bwe_i f(x_i)=f(\bwe_i x_i)$. We call $\bd{Frm}_{\ca{E}}$ the category of frames with exact maps. 
\begin{proposition}[\cite{suarez24}, Proposition 4.14]\label{whenelifts}
    A frame map $f:L\ra M$ is exact if and only if the morphism can be extended to a morphism 
        \[
        f_{\ca{E}}:(L,\opp{\fe(L)})\ra (M,\opp{\fe(M)}).
        \]
\end{proposition}
We say that a filter is \emph{regular} if it is a regular element in the frame of filters (that is, if it is of the form $F\ra \{1\}$ for some filter $F$). We call $\mathsf{Filt}_{\ca{R}}(L)$ the ordered collection of regular filters. Note that $\fr(L)\se \Fi(L)$ is the Booleanization of the frame of $\Fi(L)$. Regular filter also have a characterization in \cite{jakl24}.

\begin{proposition}\label{reg}
    The regular filters coincide with the intersections of filters of the form $\{x\in L:x\ve a=1\}$ for some $a\in L$. These are the Booleanization of $\Fi(L)$.
\end{proposition}

 One of the main theorems of \cite{jakl24} is the following. Here, $\fcp(L)$ is the collection of completely prime filters and $\fso(L)$ that of Scott-open filters, and $\ca{I}(-)$ denotes closure under set-theoretical intersections. Note that this includes the empty intersection, namely the whole frame $L$. 
\begin{theorem}\label{posetofsubloc}
For any frame $L$, we have the following poset of sublocale inclusions:
\[
\begin{tikzcd}
\fr(L)
\ar[r,"\se"]
&\mathsf{Filt}_{\mathcal{E}}(L)
\ar[dr,"\se"]\\
&
&
\mathsf{Filt}_{\mathcal{SE}}(L).
\\
\ca{I}(\mathsf{Filt}_{\mathcal{CP}}(L))
\ar[r,"\se"]
& \ca{I}(\mathsf{Filt}_{\mathcal{SO}}(L))
\ar[ur,"\se"]
\end{tikzcd}
\]
\end{theorem}

In \cite{jakl24}, the following, too, is proven.

\begin{proposition}\label{famouschar}
For a frame $L$, we have:
\begin{itemize}
    \item $L$ is pre-spatial if and only if $\ca{I}(\fso(L))$ contains all principal filters;
    \item $L$ is spatial if and only if $\ca{I}(\fcp(L))$ contains all principal filters;
    \item $L$ is subfit if and only if $\fr(L)$ contains all principal filters.
\end{itemize}
\end{proposition}

\subsection{Raney extensions}

This paper is a continuation of \cite{suarez24}: for self-containedness, nonetheless, we report in this subsection all the results from there that we are going to use. A \emph{Raney extension} is a pair $(L,C)$ such that $C$ is a coframe, $L\se C$ is a frame which meet-generates $C$ and such that the subset inclusion preserves the frame operations as well as strongly exact meets. By Proposition \ref{p:pres-se}, for a space $X$ the pair $(\Om(X),\ca{U}(X))$ always is a Raney extension, and this is the main motivating example behind the definition. Several other structures that we have seen are Raney extensions, including:

    \begin{itemize}
    \item The pair $(L,\opp{\fse(L)})$ for any frame $L$;
    \item The pair $(L,\opp{\fe(L)})$ for any frame $L$;
    \item The pair $(L,\opp{\fr(L)})$ for subfit $L$;
    \item The pair $(L,\opp{\ca{I}(\fso(L))})$ for pre-spatial $L$;
    \item The pair $(L,\opp{\ca{I}(\fcp(L))})$ for spatial $L$.
\end{itemize}

Here, we have identified each element of $L$ which its principal filter, a convention which we will continue to use without mention. Because of the isomorphisms in Theorem \ref{eandse}, for any frame $L$ the following embeddings into coframes are Raney extensions, up to isomorphism.
\begin{itemize}
    \item $\op:L\to \So(L)$, 
    \item $\cl:L\to \opp{\Sc(L)}$.
\end{itemize}

For any Raney extension $(L,C)$, by the universal property of the ideal completion of a distributive lattice, there is a coframe surjection $\bwe:\opp{\Fi(L)}\to C$. This map has a left adjoint, and this acts as $c\mapsto \up c\cap L$. From now on, we will denote this map as $\up^L:C\to \opp{\Fi(L)}$. The fixpoints on the coframe of filters are then exactly those of the form $\up^L c$ for some $c\in C$. These form a subcolocale, which we will hereon call $C^*\se \opp{\Fi(L)}$. On the other hand, all elements of $C$ are fixpoints.

\begin{theorem}\label{charC*}
For a Raney extension $(L,C)$, we have an adjunction $\bwe:\opp{\Fi(L)}\lra C:\up^L$, which restricts to a pair of mutually inverse isomorphisms $\bwe:C^*\lra C:\up^L$. These are also isomorphisms of Raney extensions $\bwe:(L,C^*)\lra (L,C):\up^L$.
\end{theorem}

\begin{corollary}\label{c:=*iso}
    If $(L,C)$ and $(L,D)$ are Raney extensions such that $C^*=D^*$, then the identity on $L$ extends to an isomorphism $(L,C)\cong (L,D)$.
\end{corollary}
\begin{proof}
    Suppose that $L$ is a frame, and that $(L,C)$ and $(L,D)$ are Raney extensions such that $C^*=D^*$. We then $(L,C^*)=(L,D^*)$. Consider, then, the composition of the isomorphisms $\up^L:(L,C)\to (L,C^*)$ and $\bwe:(L,D^*)\to (L,D)$. This, indeed, is an isomorphism which restricts to the identity on $L$.
\end{proof}

For a frame $L$ and a collection of filters $\ca{F}\se \Fi(L)$, we define the following properties for a Raney extension $(L,C)$:
\begin{itemize}
    \item \emph{$\ca{F}$-density}: each element of $C$ is a join of elements of the form $\bwe F$ for some $F\in \ca{F}$;
    \item \emph{$\ca{F}$-compactness}: for all $F\in \ca{F}$ we have $\bwe F\leq a$ implies $a\in F$ for all $a\in L$.
\end{itemize}
Similar properties are introduced for more general extensions of frames in \cite{jakl24}. These are terms are generalizations, and adaptations to the frame case, of the properties \emph{density} and \emph{compactness} of the canonical extension of a distributive lattice: see \cite{gehrke94}, \cite{gehrke01}, and \cite{gehrke08}. A Raney extension is \emph{$\ca{F}$-canonical} if it is both $\ca{F}$-dense and $\ca{F}$-compact. For a Raney extension $(L,C)$, we have that the elements $C^*$ form a closure system on $\Fi(L)$, whose associated closure operator is the composition of the two adjoints $\up^L\circ \bwe$. With this in mind, for a collection $\ca{F}$ of filters we define $\ca{F}^*=\{\up^L \bwe F:F\in \ca{F}\}$.
\begin{proposition}\label{C*canonical}
    For any Raney extension $(L,C)$ and any collection $\ca{F}\se \Fi(L)$,
    \begin{enumerate}
        \item $(L,C)$ is $\ca{F}$-dense if and only if $C^*\se \ca{I}(\ca{F}^*)$;
        \item $(L,C)$ is $\ca{F}$-compact if and only if $\ca{F}\se C^*$.
    \end{enumerate}
    In particular, $(L,C)$ is $\ca{F}$-canonical if and only if $\ca{I}(\ca{F})^{op}=C^*$.
\end{proposition}

For brevity, in the following we will refer to $\fso(L)$-canonicity simply as $\ca{SO}$-canonicity, and analogously for all other similarly denoted collections of filters, and for density and compactness.

\begin{corollary}\label{c:all-e-comp}
    All Raney extensions are $\ca{E}$-compact and $\ca{R}$-compact.
\end{corollary}
\begin{proof}
    For a Raney extension $C^*$, all principal filters of $L$ are in $C^*$, and by Lemma \ref{principalmin}, we must have $\fe(L)\se \ca{U}(X)^*$. The result follows by the characterization in Proposition \ref{C*canonical}. Since $\fr(L)\se \fe(L)$, $\ca{E}$-compactness implies $\ca{R}$-compactness.
\end{proof}

In \cite{suarez24} we use the theory of polarities of Birkhoff to show that, under certain conditions, $\ca{F}$-canonical Raney extensions exist, and that in that case they are unique. Note that a result similar to the following, for more general filter extensions of frames, is proven in \cite{jakl24}.

\begin{theorem}[\cite{suarez24}, Theorem 3.6, see also \cite{jakl24}, Proposition 4.3]\label{containsprincipal}
    For a frame $L$ and any collection $\ca{F}\se \Fi(L)$ of filters, the $\ca{F}$-canonical Raney extension exists if and only if:
    \begin{enumerate}
    \item $\ca{I}(\ca{F})$ contains all principal filters;
    \item $\ca{I}(\ca{F})^{op}\se \Fi(L)^{op}$ is a subcolocale inclusion;
    \item All filters in $\ca{F}$ are strongly exact.
    \end{enumerate}
    In case it exists, it is unique, up to isomorphism. Concretely, it is described as the structure coming from the theory of polarities of Birkhoff, that is, the pair $(L,\opp{\ca{I}(\ca{F})})$.
\end{theorem}
A morphism $f:(L,C)\to (M,D)$ of Raney extensions is a coframe map $f:C\to D$ such that $f(a)\in M$ whenever $a\in L$ and such that it preserves the frame operations of $L$. The category of Raney extensions is called $\bd{Raney}$. Some frame maps can be extended to morphisms between Raney extensions on them.

\begin{theorem}\label{whenliftsfilters}\label{whenlifts}
    Suppose that $f:L\ra M$ is a frame map and that we have Raney extensions $(L,C)$ and $(M,D)$. The frame map extends to a map of Raney extensions if and only if $f^{-1}(F)\in C^*$ for every $F\in D^*$. If this map exists, it is $\clo{D^*}(f[-])$. 
\end{theorem}

There is a natural order on Raney extensions on a frame $L$, that is, subcolocale inclusion of their coframe component. If we order them this way, we obtain the following.
\begin{theorem}\label{Rboundaries}
    For a frame $L$, the collection of Raney extensions over $L$ is isomorphic to the section $[\fe(L),\fse(L)]$ of the coframe of sublocales of $\fse(L)$.
\end{theorem}

\section{Spectra of Raney extensions}\label{S2}

In this section, we extend the category of $T_0$ topological spaces to a pointfree category. We do so by showing that there is an adjunction between $\opp{\bd{Raney}}$ and $\bd{Top}$. Firstly, we will define the spectrum of a Raney extension. For any coframe $C$, we define $\rpt(C)$ to be the collection of its completely join-prime elements. For a Raney extension $(L,C)$, let us define the function $\varphi_{(L,C)}:C\ra \ca{P}(\rpt(C))$ as
\[
\varphi_{(L,C)}(a)=\{x\in \rpt(C):x\leq a\}.
\]
It is easy to see that the following two properties hold:
\begin{enumerate}
    \item $\varphi_{(L,C)}(\bwe_i a_i)=\bca_i\varphi_{(L,C)}( a_i)$,
    \item $\varphi_{(L,C)}(\bve_i a_i)=\bigcup_i \varphi_{(L,C)}(a_i)$,
\end{enumerate}
for each family $a_i\in L$. When the Raney extension $(L,C)$ is clear from the context, we will often abbreviate $\varphi_{(L,C)}$ as simply $\varphi$. By property 2, we have that the elements of the form $\varphi_{(L,C)}(a)$ for $a\in L$ form a topology. We denote the topological space obtained by equipping the set $\rpt(C)$ with this topology as $\rpt(L,C)$, and we call it the \emph{spectrum} of the Raney extension $(L,C)$. Since all elements of $C$ are meets of elements of $L$, from property 1 it follows that the elements of the form $\varphi_{(L,C)}(c)$ with $c\in C$ are the saturated sets of this space. Let us show functoriality of the assignment $(L,C)\mapsto \rpt(L,C)$. 

\begin{lemma}\label{cpcjp}
For a Raney extension $(L,C)$, an element $x\in C$ is completely join-prime if and only if $\up^{L} x$ is a completely prime filter.
\end{lemma}
\begin{proof}
It is immediate that if $x\in C$ is completely join-prime then $\up^L x$ is completely prime. For the converse, suppose that we have $x\in C$ such that $\up^{L}x$ is completely prime. Suppose that $x\leq \bve D$ for $D\se C$. This means that $\up^L \bve D\se \up^L x$. Observe that $\up^L \bve D=\bigcap \{\up^L d:d\in D\}$. As $\up^L x$ is assumed to be completely prime, there must be some $d\in D$ such that $\up^L d\se \up^L x$. This implies that $x\leq d$.
\end{proof}

\begin{lemma}\label{respectscjp}
For a morphism $f:(L,C)\ra (M,D)$ of Raney extensions, if $x\in \rpt(D)$ then $f^*(x)\in \rpt(C)$. 
\end{lemma}
\begin{proof}
By Lemma \ref{cpcjp}, it suffices to show that for a morphism $f:(L,C)\ra (M,D)$ of Raney extensions, if $x\in \rpt(D)$ then $\up^L f^*(x)$ is a completely prime filter of $L$. If $f^*(x)\leq \bve A$ for $A\se L$, then as $f$ respects the frame operations of $L$, and because $f^*\dashv f$, we have that $x\leq \bve \{f(a):a\in A\}$. Since $x$ is completely join-prime, there is some $a\in A$ such that $x\leq f(a)$, that is $f^*(x)\leq a$.
\end{proof}

\begin{lemma}
The assignment $\rpt:(L,C)\mapsto \rpt(L,C)$ is the object part of a functor $\rpt:\opp{\bd{Raney}}\ra \bd{Top}$ which acts on morphisms as $f\mapsto f^*$.
\end{lemma}
\begin{proof}
That every morphism is mapped to a well-defined function between the set of points follows from Lemma \ref{respectscjp}. Continuity follows from the fact that the $f^*$-preimage of some $\varphi(a)$ for $a\in L$ is, expanding definitions,
\begin{gather*}
   \{x\in \rpt(D):f^*(x)\leq a\}=\\
   \{x\in \rpt(D):x\leq f(a)\}=\varphi(f(a)),
\end{gather*}
and this set is indeed open in $\rpt(D)$ as by definition of Raney morphism $f(a)\in M$.
\end{proof}

By Theorem \ref{charC*}, for every Raney extension $(L,C)$, its coframe component $C$ can be identified with a collection of filters of $L$. Let us now see how to describe the spectrum of a Raney extension, under this identification. 

\begin{theorem}\label{t:pointsfilters}
    For a frame $L$ and for a sublocale $\ca{F}\se \Fi(L)$ such that it contains all principal filters, we have $\pt_R(\ca{F}^{op})=\fcp(L)\cap \ca{F}$. 
    
\end{theorem}
\begin{proof}
   We show that an element $P\in \ca{F}$ is completely prime in the frame $\ca{F}$ if and only if it is completely prime as an element of $\Fi(L)$. If an element $P\in \ca{F}$ is completely prime in the frame $\Fi(L)$, then it is also completely prime as an element of $\ca{F}$, as meets of elements of $\ca{F}$ are a subset of all the meets in $\Fi(L)$. For the converse, suppose that $P$ is completely prime in $\ca{F}$, and that $\bve_i x_i\in P$ for some collection $x_i\in L$. This means $\bca_i \up x_i\se P$, and because $\ca{F}$ contains all principal filters and by assumption on $P$, we have $\up x_i\se P$ for some $i\in I$. 
\end{proof}

\begin{corollary}\label{c:pointsfilters}
   A Raney extension $(L,C^*)$ has as points the elements of $C^*\cap \fcp(L)$, and as opens the sets of the form $\{P\in \fcp(L)\cap C^*:a\in P\}$ for some $a\in L$.
\end{corollary}
\begin{proof}
    The first part of the statement is a direct consequence of Theorem \ref{t:pointsfilters}. For the second part of the statement, it suffices to unravel the definition of the topology on $\rpt(L,C^*)$.
\end{proof}

We now define the left adjoint to $\rpt$. For a topological space $X$ we define $\Om_R(X)$ as the pair $(\Om(X),\ca{U}(X))$, we extend the assignment to morphisms as $f\mapsto f^{-1}$.

\begin{lemma}\label{raneyrefl}
For every Raney extension $(L,C)$ there is a surjective map of Raney extensions $\varphi_{(L,C)}:(L,C)\ra \Om_R(\rpt(L,C))$. This is an isomorphism precisely when $C$ is join-generated by its completely join-prime elements.
\end{lemma}
\begin{proof}
    The fact that it is a surjection and a map of Raney extensions follows from properties 1 and 2 of the topologizing map $\varphi_{(L,C)}$. The map is an isomorphism precisely when it is injective, and this happens exactly when for $c,d\in C$ such that $c\nleq d$ there is some $x\in \pt_R(C)$ such that $x\leq c$ and $x\nleq d$. This holds if and only if the completely join-prime elements join-generate $C$.
\end{proof}

The map we have just defined will be the evaluation at an object of the natural transformation $\Om_R\circ \rpt\Rightarrow 1_{\opp{\bd{Raney}}}$. Let us now define the other natural transformation $1_{\bd{Top}}\Rightarrow \rpt\circ \Om_R$.
\begin{lemma}\label{toprefl}
For every topological space $X$ the map $\psi_X:X\ra \rpt(\Om_R(X))$ defined as $x\mapsto \up x$ is a continuous map. This is a homeomorphism precisely when $X$ is a $T_0$ space.
\end{lemma}
\begin{proof}
That the map is well-defined and surjective follows from the observation that the completely join-prime elements of $\ca{U}(X)$ are precisely the principal upsets. For continuity, we observe that the $\psi_X$-preimage of an open set $\varphi(U)$ is the set $\{x\in X:\up x\in \varphi(U)\}=U$. This map is also open, as the direct image of an open $U\se X$ is the open $\{\up x:\up x\se U\}=\varphi(U)$. The map is then a homeomorphism when it is injective, and this holds if and only if whenever $x\neq y$ we have $\up x\neq \up y$. This amounts to the specialization preorder being an order, that is, the space being $T_0$. 
\end{proof}

In the following, for a space $X$ and for $x\in X$, we denote the neighborhood filter of $x$ in $\Om(X)$ as $N(x)$.
\begin{lemma}\label{l:N(x)-gen}
For a space $X$, we have $\ca{U}(X)^*=\ca{I}(\{N(x):x\in X\})$.
\end{lemma}
\begin{proof}
   We notice that for each $x\in X$ we have $\up^{\Om(X)}\up x=N(x)$. As the elements of the form $\up x$ join-generate $\ca{U}(X)$, for each $U\in \ca{U}(X)$ we have that $\up^{\Om(X)}U=\bca \{N(x):x\in U\}$. 
\end{proof}
Recall that an adjunction $L:\ca{C}\lra \ca{D}:R$ is said to be \emph{idempotent} if every element of the form $R(d)$ for some object $d\in \mf{Obj}(\ca{D})$ is a fixpoint on the $\ca{C}$ side, and the same holds for the $\ca{D}$ side. 
\begin{theorem}
    The pair $(\Om_R,\rpt)$ constitutes an idempotent adjunction $\bd{Top}\lra \opp{\bd{Raney}}$. Raney duality is the restriction of this adjunction to a dual equivalence.
\end{theorem}
\begin{proof}
We claim that the two maps defined in Lemmas \ref{raneyrefl} and \ref{toprefl} are the required natural transformations, as defined for each component. Suppose that $(L,C)$ is a Raney extension, and that we have a map $f:(L,C)\ra (\Om(X),\ca{U}(X))$ in $\bd{Raney}$ for some space $X$. Let will define a map $f_{\varphi}$ such that the following commutes:
\begin{center}
    \begin{tikzcd}[row sep=large,column sep=large]
        & (\varphi[L],\varphi[C])
        \ar[dr,"{f_{\varphi}}"]
        \\
        (L,C)
        \ar[ur,"\varphi"]
        \ar[rr,"f"]
        && (\Om(X),\ca{U}(X)).
    \end{tikzcd}
\end{center}
We slightly abuse notation and define a frame map $f_{\varphi}:\varphi[L]\to \ca{U}(X)$ as forced by commutativity of the diagram, namely, as $\varphi(a)\mapsto f(a)$. By Theorem \ref{whenliftsfilters}, to show that this can be extended to a map of Raney extension, it suffices to show that preimages of filters in $\ca{U}(X)^*$ are in $\varphi[C]^*$. By Lemma \ref{l:N(x)-gen} above, it suffices to show the claim for filters of the form $N(x)$. Let $x\in X$. We have 
\[
f_{\varphi}^{-1}(N(x))=\{\varphi(a):x\in f(a)\}=\{\varphi(a):f^*(\up x)\leq a\}=\{\varphi(a):\varphi(f^*(\up x))\se \varphi(a)\},
\]
where $f^*(\up x)\leq a$ is equivalent to $\varphi(f^*(\up x))\se \varphi(a)$ because $f^*(\up x)$ is completely join-prime. For spaces, consider a space $X$ and a Raney extension $(L,C)$, and suppose that there is a continuous map $f:X\ra \rpt(L,C)$. We define the map $f_{\psi}$ making the following commute.
\begin{center}
    \begin{tikzcd}[row sep=large,column sep=large]
        & \rpt(\Om(X),\ca{U}(X))
        \ar[dr,"{f_{\psi}}"]
        \\
        X
        \ar[ur,"\psi_X"]
        \ar[rr,"f"]
        && \rpt(L,C).
    \end{tikzcd}
\end{center}
For a completely join-prime element $\up x$, we define $f_{\psi}(\up x)=f(x)$. Routine calculations show that this map is continuous. Let us see that the adjunction is idempotent. By Lemma \ref{raneyrefl}, any Raney extension $(\Om(X),\ca{U}(X))$ is a fixpoint, as the coframe $\ca{U}(X)$ is join-generated by the elements of the form $\up x$. By Lemma  \ref{toprefl}, any $T_0$ space is a fixpoint.
\end{proof}
Motivated by the result above and by Lemma \ref{raneyrefl}, we say that a Raney extension $(L,C)$ is \emph{spatial} if $C$ is join-generated by the completely join-prime elements.

\begin{proposition}\label{charspatiality}
A Raney extension $(L,C)$ is spatial if and only if $C^*\se \ca{I}(C^*\cap \fcp(L))$.
\end{proposition}
\begin{proof}
Because of the isomorphism $\up^L:C\cong C^*$, a Raney extension $(L,C)$ is spatial precisely when all elements of $C^*$ are intersections of completely join-prime elements in $C^*$, as by Corollary \ref{c:pointsfilters} we have $\rpt(C^*)=\fcp(L)\cap C^*$.
\end{proof}

\begin{proposition}
    For a spatial frame $L$, the pair $(L,\opp{\ca{I}(\fcp(L))})$ is the free spatial Raney extension over it. In particular, the category of spatial frames is a full coreflective subcategory of that of spatial Raney extensions.
\end{proposition}
\begin{proof}
The assignment $L\mapsto (L,\opp{\ca{I}(\fcp(L))})$ from the category of spatial frames to $\bd{Raney}$ can be extended to morphisms, by Theorem \ref{whenliftsfilters}. The assignment, then, is functorial. Suppose that we have a map $f:L\ra M$ between spatial frames, and that $(M,C)$ is a spatial Raney extension. By spatiality, we must have $C^*\se \ca{I}(\fcp(M))$, by Proposition \ref{charspatiality}. Preimages under $f$ of completely prime filters are completely prime. This means that preimages of filters in $C^*$ are in $\ca{I}(\fcp(L))$. By Proposition \ref{whenelifts}, we have a morphism $(L,\opp{\ca{I}(\fcp(L))})\ra(M,C)$ which extends the frame map $f:L\ra M$. 
\end{proof}

\subsection{The collection of Raney spectra on a frame}

In this subsection, our final goal is proving that, on a frame $L$, for any Raney extension $(L,C)$ we have subspace inclusions $\pt_D(L)\se \rpt(L,C)\se \pt(L)$.

\begin{lemma}\label{bigmeet}
For a frame $L$, for any $a\in L$ the meet $\bwe \{x\in L:a<x\}$ is exact.
\end{lemma}
\begin{proof}
Let $L$ be a frame and let $a\in L$. Let us consider the meet $\bwe \{x\in L:a<x\}$. Let $b\in L$. We claim that $\bwe \{x\ve b:a<x\}\leq \bwe \{x\in L:a<x\}\ve b$. We consider two cases. First, let us assume that $b\leq a$. If this is the case, then $b\leq x$ whenever $a<x$, and so both the left hand side and the right hand side equal $\bwe \{x\in L:a<x\}$. Now, let us assume instead that $b\nleq a$. This is equivalent to saying that $a<a\ve b$. This means that we have the chain of inequalities 
\[
\bwe \{x\ve b:a<x\}\leq a\ve b\leq \bwe \{x\in L:a<x\}\ve b.\qedhere
\]
\end{proof}
\begin{proposition}\label{p:cpexact}
A completely prime filter $L{\sm}\da p$ is exact if and only if the prime $p$ is covered.
\end{proposition}
Suppose that the completely prime filter $L{\sm}\da p$ is exact. To show that the prime $p$ is covered, we prove that $\bwe \{x\in L:p<x\}\nleq p$. By Lemma \ref{bigmeet}, the meet on the left-hand side is exact. The result follows by our assumption that $L{\sm}\da p$ is closed under exact meets. For the converse, we suppose that $p$ is a covered prime and that $x_i\nleq p$ for the members of some family $\{x_i:i\in I\}$ such that their meet is exact. We then have that $x_i\ve p\neq p$ for every $i\in I$, and as $p$ is covered, this implies that $\bwe_i (x_i\ve p)\neq p$. By exactness of the meet $\bwe_i x_i$, we also have $(\bwe_i x_i)\ve p\neq p$, that is $\bwe _i x_i\nleq p$, as required.

\begin{lemma}\label{spectrumofef}
For any frame $L$, the spectrum of $(L,\opp{\fe(L)})$ is homeomorphic to the space $\pt_D(L)$. The spectrum of $(L,\opp{\fse(L)})$ is the classical spectrum $\pt(L)$.
\end{lemma}
\begin{proof}
By Corollary \ref{c:pointsfilters}, the points of $(L,\opp{\fe(L)})$ are the completely prime filters which are also exact. By Proposition \ref{p:cpexact}, these are the filters of the form $L{\sm}\da p$ for some covered prime $p\in L$. Indeed, then, we have a bijection between the points of $\rpt(L,\opp{\fe(L)})$ and those of $\pt_D(L)$. This is a restriction of the standard homeomorphism between the spectrum $\pt(L)$ and its space of completely prime filters, and so it is a homeomorphism. For $(L,\opp{\fse(L)})$, it suffices to notice that since all completely prime filters are strongly exact, $\fcp(L)\cap \fse(L)=\fcp(L)$.
\end{proof}
We shall now refine the result above to the case of subfit frames. We call $\maxpt(L)$ the collection of maximal primes of a frame $L$, equipped with the subspace topology inherited from $\pt(L)$.

\begin{proposition}\label{minimalcpf}
    Let $L$ be a frame. A prime $p\in L$ is maximal if and only if $L{\sm}\da p$ is a regular filter.
\end{proposition}

\begin{proof}
Suppose that we have a maximal prime $p\in L$. Because it is maximal, we have $\up p=\{p,1\}$. We claim that the completely prime filter $L{\sm}\da p$ is its pseudocomplement in the frame of filters. Indeed, we have $L{\sm}\da p\cap \{1,p\}=\{1\}$. Furthermore, if for a filter $F$ we have $F\cap \{1,p\}=\{1\}$ then $p\notin F$, and so for $f\in F$ we must have $f\nleq p$. For the converse, suppose that we have a prime $p\in L$ such that $L{\sm}\da p$ is a regular filter. By Proposition \ref{reg}, this is the intersection of a collection of filters of the form $\{x\in L:x\ve a=1\}$ for some $a\in L$. As $L{\sm}\da p$ is completely prime, it must be $\{x\in L:x\ve a=1\}$ for some $a\in L$. This means that for all $x\in L$ the conditions $x\leq p$ and $x\ve a\neq 1$ are equivalent. In particular, because the filter is not all of $L$ (as it is completely prime), we must have $a\leq p$ since $a\ve a=a\neq 1$. This means that if $x\nleq p$ then $x\ve a=1$ and so $x\ve p=1$, for all $x\in L$. This means that $p$ must be maximal. 
\end{proof}

\begin{proposition}
For a subfit frame $L$, the spectrum of the Raney extension $(L,\opp{\fr(L)})$ is the $T_1$ space $\maxpt(L)$.
\end{proposition}
\begin{proof}
Suppose that $L$ is a subfit frame. We claim that all its exact filters are regular. By Proposition \ref{famouschar}, we have that $\fr(L)$ contains all principal filters, and so by Lemma \ref{principalmin} we must have $\fe(L)\se \fr(L)$. The reverse inclusion holds for all frames. By Corollary \ref{c:pointsfilters}, then, the points of $(L,\opp{\fe(L)})$ are the regular completely prime filters, which by Proposition \ref{minimalcpf} are those corresponding to maximal primes of $L$. The fact that this is a homeomorphism comes from the fact that this is a restriction of the standard homeomorphism between the spectrum $\pt(L)$ and the spectrum defined in terms of prime elements of $L$. The space $\maxpt(L)$ is a $T_1$ space, since whenever $p,q\in \maxpt(L)$ we have both $p\nleq q$ and $q\nleq p$ by maximality, and so the open set $\{a\in L:a\nleq p\}$ contains $q$ and omits $p$, and the open set $\{a\in L:a\nleq q\}$ contains $p$ and omits $q$. 
\end{proof}

\begin{lemma}\label{spectrumboundary}
For a Raney extension $(L,C)$ we have subspace inclusions 
\[
\pt_D(L)\se \rpt(L,C)\se \pt(L).
\]  
\end{lemma}
\begin{proof}
Suppose that $(L,C)$ is a Raney extension. We have $\fe(L)\se C^*\se \fse(L)$, by Theorem \ref{containsprincipal} and by Lemma \ref{principalmin}. Therefore, we also have
\[
\fcp(L)\cap \fe(L)\se\fcp(L)\cap C^*\se\fcp(L)\cap \fse(L).
\]
By Corollary \ref{c:pointsfilters}, this means that we have a chain of subspace inclusions
\[
\rpt(L,\opp{\fe(L)})\se \rpt(L,C)\se \rpt(L,\opp{\fse(L)}).
\]
The result follows from Lemma \ref{spectrumofef}.
\end{proof}

\begin{lemma}\label{meetofheyting}
For a frame $L$ and a subset $\ca{X}\se \Fi(L)$ we have that the smallest sublocale $\ca{S}(\ca{X})$ is the set $\ca{I}(\{\up a\ra F:a\in L,F\in \ca{X}\})$.
\end{lemma}
\begin{proof}
By Lemma \ref{l:small-sl}, it suffices to show that the collection in the claim is the same as $\ca{I}(\{G\ra F:G\in \Fi(L),F\in \ca{X}\})$. Indeed, for each $G\in \Fi(L)$ and $F\in \ca{X}$, we have $G\ra F=\bca \{\up g\ra F:g\in G\}$.
\end{proof}
The following fact follows directly from Lemma \ref{l:prime}, and the fact that completely prime filters are prime elements of $\Fi(L)$.
\begin{lemma}\label{l:heytingcp}
    For a frame $L$ and a completely prime filter $P\se L$, for each $a\in L$ we have $\up a\to P=L$ if $a\in P$, and $\up a\ra P=P$ otherwise.
\end{lemma}

In the following, for brevity, we will identify prime elements with the corresponding completely prime filters, without mention.

\begin{theorem}\label{spectrumpowerset}
The spectra of Raney extensions over $L$ coincide with the interval
\[
[\pt_D(L),\pt(L)] 
\]
of the powerset of $\pt(L)$.
\end{theorem}
\begin{proof}
Recall that, by Lemma \ref{spectrumofef}, we have $\pt_D(L)=\fe(L)\cap \fcp(L)$. That, for a Raney extension $(L,C)$, its spectrum is contained in the $[\pt_D(L),\pt(L)]$ interval is the content of Lemma \ref{spectrumboundary}. Conversely, suppose that we have a collection of completely prime filters $\ca{P}\se \Fi(L)$ such that $\fe(L)\cap \fcp(L)\se \ca{P}$. Consider the sublocale $\ca{S}(\ca{P}\cup L)\se \Fi(L)$. By Lemma \ref{principalmin}, this is the same as $\ca{S}(\ca{P}\cup \fe(L))$. Observe that $\fe(L)$ is stable under $\up a\to -$ for each $a\in L$, as it is a sublocale. The same holds for $\ca{P}$, by Lemma \ref{l:heytingcp}. By Lemma \ref{meetofheyting}, $\ca{S}(\ca{P}\cup L)=\ca{I}(\ca{P}\cup \fe(L))$. We now consider the Raney extension $(L,\opp{\ca{S}(\ca{P}\cup L)})$. It is clear that all the elements of $\ca{P}$ are points of this Raney extension, by Corollary \ref{c:pointsfilters}. Let us show the reverse set inclusion. Suppose that there is a completely prime filter $F$ such that $F\in \ca{S}(\ca{P}\cup L)$. By complete primality, and by the characterization above, this is either in $\ca{P}$ or in $\fe(L)$. In the second case, it is in $\ca{P}$, too, by assumption on $\ca{P}$. Indeed, then, $\rpt(L,\opp{\ca{S}(\ca{P}\cup L)})=\ca{P}$, as desired.
\end{proof}

\begin{remark}
It may be surprising that the spectrum $\pt_R(L,C)$ does not contain all points of $\pt(L)$, as this may be seen as a spectrum construction that forgets about too much information. However, it is the coframe $C$ that ought to be seen, alone, as the ordered structure of which we are taking the points. The frame $L$ (just like in Raney duality) is nothing but a carrier of information on how to topologize such set of points. Furthermore, from the result above, we may see that this is what makes Raney extensions more expressive than frames: if all points of $L$ were points of $\rpt(L,C)$, then Raney extensions would only be able to capture the sober spaces. 
\end{remark}

\section{Topological properties and Raney extensions}

\subsection{Sobriety}
We have the following result of \cite{suarez24}.
\begin{proposition}
    A space $X$ is sober if and only if $(\Om(X),\ca{U}(X))$ is $\ca{CP}$-compact.
\end{proposition}
Motivated by this, we define a Raney extension $(L,C)$ to be \emph{sober} if it is $\ca{CP}$-compact. Note that by the characterization in Proposition \ref{C*canonical}, this is equivalent to saying that every completely prime filter of $L$ is $\up^L x$ for some $x\in C$, which is then necessarily completely join-prime by Lemma \ref{cpcjp}. Thus, for spatial Raney extensions, our definition of sobriety is equivalent to that in \cite{bezhanishvili20}. We work towards proving that any Raney extension admits a \emph{sobrification}, a completion to a sober space. For a Raney extension $(L,C)$ we call a map $\sigma:S(L,C)\ra (L,C)$ of the category $\bd{Raney}$ a \emph{sobrification} if $S(L,C)$ is sober, and if whenever $f:(M,D)\ra (L,C)$ is a morphism from a sober Raney extension, we have a commuting diagram
\[
\begin{tikzcd}
S(L,C)
\ar[r,"\sigma"]
& (L,C).\\
(M,D)
\ar[u,"f_{\sigma}"]
\ar[ur,swap,"f"]
\end{tikzcd}
\]

\begin{theorem}
For a Raney extension $(L,C)$, the map
\begin{gather*}
\sigma:(L,\opp{\ca{I}(C^*\cup \fcp(L))})\ra (L,C)\\
F\mapsto \bwe F
\end{gather*}
is its sobrification.
\end{theorem}
\begin{proof}
Observe that, as $C^*\se \Fi(L)$ is a sublocale and by Lemmas \ref{meetofheyting} and \ref{l:heytingcp}, we have that $\ca{I}(C^*\cup \fcp(L))$ is a sublocale. As $(L,\opp{\ca{I}(C^*\cup \fcp(L))})$ contains all completely prime filters of $L$, indeed, by Proposition \ref{C*canonical} it is $\ca{CP}$-compact. Since $C^*\se \ca{I}(C^*\cup \fcp(L))$, by Theorem \ref{whenliftsfilters} it means that the identity on $L$ extends to a surjective map of Raney extensions
\begin{gather*}
    \sigma:(L,\opp{\ca{I}(C^*\cup \fcp(L))})\to (L,C)\\
    F\mapsto \bwe F.
\end{gather*}
Let us show that this map has the required universal property. Suppose that $f:(M,D)\ra (L,C)$ is a Raney map from a sober Raney extension. We then have a frame map $f|_{M}:M\ra L$. By Theorem \ref{whenliftsfilters}, to show that the map lifts it suffices to show that the preimage of each filter in $\fcp(L)$ as well as each filter in $C^*$ is in $D^*$. For filters in $C^*$, this holds because there is a map $f:(M,D)\ra (L,C)$. For a completely prime filter $P\se L$, recall that we have $f^{-1}(P)\in \fcp(M)$, as preimages of completely prime filters are completely prime, and by definition of sobriety and Proposition \ref{C*canonical}, we also have $\fcp(M)\se D^*$. This map extends the frame map $f|_M$. Thus, the diagram commutes as desired.
\end{proof}

Let us now compare sobriety with spatiality for Raney extensions.

\begin{lemma}\label{cpcanonical}
    A Raney extension $(L,C)$ is sober and spatial if and only if it is $\ca{CP}$-canonical.
\end{lemma}
\begin{proof}
It follows from Proposition \ref{charspatiality} and by Proposition \ref{C*canonical} that a Raney extension $(L,C)$ is sober and spatial if and only if $C^*=\ca{I}(\fcp(L))$. This holds if and only if the Raney extension is $\ca{CP}$-canonical. 
\end{proof}

\begin{proposition}
A spatial frame $L$ admits a unique sober and spatial Raney extension, up to isomorphism. This is the free spatial Raney extension $(L,\opp{\ca{I}(\fcp(L))})$.
\end{proposition}
\begin{proof}
   By Lemma \ref{cpcanonical}, when a sober and spatial Raney extension exists, it is unique, up to isomorphism, by Theorem \ref{containsprincipal}. If $L$ is a spatial frame, then $(L,\opp{\ca{I}(\fcp(L))})$ is a Raney extension by Proposition \ref{famouschar}, and it is the $\ca{CP}$-canonical Raney extension by Theorem \ref{containsprincipal}.
\end{proof}

\subsection{The \texorpdfstring{T\textsubscript{D}}{TD} axiom}

Let us now look at the Raney analogue of the $T_D$ axiom. 

\begin{lemma}\label{l:N(x)exact}
    A $T_0$ space is $T_D$ if and only if all neighborhood filters are exact.
\end{lemma}
\begin{proof}
    Suppose that $X$ is a $T_D$ topological space. Neighborhood filters are completely prime, and by Proposition \ref{allpxarecov} all primes of the form $X{\sm}\overline{\{x\}}$ are covered. Hence, by the characterization in Proposition \ref{p:cpexact}, the corresponding neighborhood filters are exact. Conversely, if $X$ is not $T_D$ there must be a point $x\in X$ whose prime is not covered, and by Proposition \ref{p:cpexact} again, this means that its completely prime filter is not exact.    
\end{proof}

\begin{theorem}\label{t:chartd}
The following are equivalent for a $T_0$ space $X$.
\begin{enumerate}
    \item The space $X$ is $T_D$. 
    \item The Raney extension $(\Om(X),\ca{U}(X))$ is $\ca{E}$-dense.
    \item The Raney extension $(\Om(X),\ca{U}(X))$ is $\ca{E}$-canonical.
    \item The Raney extension $(\Om(X),\ca{U}(X))$ is isomorphic to $(\Om(X),\fe(\Om(X)))$.
    \item The inclusion $\Om(X)\se \ca{U}(X)$ preserves exact meets.
\end{enumerate}
\end{theorem}
\begin{proof}
Let $X$ be a $T_0$ space. If this is a $T_D$ space, then by Lemma \ref{l:N(x)exact} all neighborhood filters are exact, and this means that all filters of the form $\up^{\Om(X)}\up x$ for $x\in X$ are exact. As for all $x\in X$ we have $\bca \up^{\Om(X)}\up x=\up x$, and the principal filters generate the collection $\ca{U}(X)$, (2) follows. Suppose, now, that (2) holds. By Corollary \ref{c:all-e-comp}, the Raney extension is $\ca{E}$-compact, hence $\ca{E}$-canonical by our initial hypothesis. Items (3) and (4) are equivalent by the uniqueness result of Theorem \ref{containsprincipal}. Suppose that (4) holds. We will identify $(\Om(X),\ca{U}(X))$ with the isomorphic Raney extension $(\Om(X),\opp{\fe(\Om(X))})$. If $U_i\in \Om(X)$ is a family such that their meet is exact, then we have that the least upper bound of the family $\up U_i$ in $\fe(\Om(X))$ must be $\up \bwe_i U_i$, by definition of exact filter. This means that the meet is preserved by the embedding $\Om(X)\to \opp{\fe(\Om(X))}$. Finally, (5) implies (1) by the characterization in Theorem \ref{t:ball}.
\end{proof}

Motivated by the last result, we call a Raney extension $T_D$ if it is $\ca{E}$-dense. All Raney extensions are $\ca{E}$-compact, by Corollary \ref{c:all-e-comp}. Thus, the $T_D$ Raney extensions are those which are $\ca{E}$-canonical, and by the uniqueness result of Theorem \ref{containsprincipal} these are the Raney extensions which are, up to isomorphism, $(L,\opp{\fe(L)})$ for some frame $L$.

\begin{proposition}
    The forgetful functor $\pi_1:\bd{Raney}\to \bd{Frm}$ restricts to an isomorphism between the category of $T_D$ Raney extensions and $\bd{Frm}_{\ca{E}}$.
\end{proposition}
\begin{proof}
    For a map $f:(L,\opp{\fe(L)})\to (M,\opp{\fe(L)})$ of $T_D$ Raney extensions, we must have by Theorem \ref{whenliftsfilters} that the restriction $f|_{L}:L\to M$ is a map in $\bd{Frm}_{\ca{E}}$. Thus, the restriction and co-restriction of $\pi_1$ is well-defined. The inverse functor maps each frame $L$ to the Raney extension $(L,\opp{\fe(L)})$, and this assignment is functorial by Proposition \ref{whenelifts}.
\end{proof}

For a Raney extension $(L,C)$, we call a \emph{$T_D$ reflection} a map $\delta:(L,C)\to D(L,C)$ such that $D(L,C)$ is $T_D$, and such that whenever $f:(L,C)\to (M,D)$ is a map to a $T_D$ Raney extension, we have a commuting diagram as follows.
\[
\begin{tikzcd}
(L,C)
\ar[r,"\delta"]
\ar[dr,"f",swap]
& D(L,C)
\ar[d,"f_{\delta}"]\\
& (M,D).
\end{tikzcd}
\]
Let us call $\bd{Raney}_{\ca{E}}$ the category of Raney extensions with maps such that their restriction to the frame components is exact. 
\begin{proposition}
    In the category $\bd{Raney}_{\ca{E}}$, every Raney extension admits a $T_D$ reflection.
\end{proposition}
\begin{proof}
   We claim that the required map for $(L,C)$ is $\delta:(L,C)\to (L,\opp{\fe(L)})$ defined as $c\mapsto \mb{cl}_{\ca{E}}(\up^L c)$.
The Raney extension $(L,\opp{\fe(L)})$ is $T_D$, by definition. The identity on $L$ is exact. By Corollary \ref{c:all-e-comp}, we have $\opp{\fe(L)}\se C^*$, and so by Theorem \ref{whenliftsfilters} the map above is a map of Raney extensions, and it is a map in $\bd{Raney}_{\ca{E}}$. Now, suppose that there is a $T_D$ Raney extension $(M,D)$ such that we have a morphism $f:(L,C)\to (M,D)$ in $\bd{Raney}_{\ca{E}}$. Because $(M,D)$ is $T_D$, we must have $D^*=\opp{\fe(L)}$. By assumption on $f$, then, the preimage map relative to $f|_{L}$ maps filters in $D^*$ to exact filters of $L$. Hence, by Theorem \ref{whenliftsfilters}, we have a map $f_{\delta}:(L,\opp{\fe(L)})\to (M,D)$ as required. Finally, this map is in $\bd{Raney}_{\ca{E}}$ as it extends $f|_L$.
\end{proof}

\subsection{The \texorpdfstring{T\textsubscript{1}}{T1} axiom}
Let us now look at the $T_1$ axiom. The axiom $T_1$, too, can be characterized in terms of filters.
\begin{lemma}\label{l:N(x)regular}
    A $T_0$ space is $T_1$ if and only if all its neighborhood filters are regular.
\end{lemma}
\begin{proof}
    Suppose that $X$ is a $T_1$ space, and let $x\in X$. As $X$ is $T_1$, the set $X{\sm}\{x\}$ is open. We have $N(x)=\{U\in \Om(X):U\cup (X{\sm}\{x\})=X\}$. By the characterization of regular filters in Proposition \ref{reg}, this is a regular filter. For the converse, suppose that $X$ is a $T_0$ space where all neighborhood filters are regular. Let $x\in X$. We will show that $\{x\}$ is closed by showing $\da x=\{x\}$. By the characterization in Proposition \ref{reg}, and because neighborhood filters are completely prime, there is some open $V\in \Om(X)$ such that: 
\[
N(x)=\{U\in \Om(X):U\cup V=X\}.
\]
Observe that $\bca N(x)\cup V=\up x\cup V=X$, thus $V^c\se \up x$. Since $\emptyset \notin N(x)$, we have $V=\emptyset \cup V\neq X$. Then, also $V\cup V\neq X$, from which $x\notin V$. As $V^c$ is a downset in the specialization order, we have $\da x\se V^c$. But this means $\da x\se V^c\se \up x$, hence $\da x=\{x\}$. 
\end{proof}

\begin{lemma}\label{maximalboolean}
For a frame $L$, the Booleanization $\bl(0)$ is maximal among the Boolean sublocales, meaning that for each $x\in L$ we have that $\bl(0)\se \bl(x)$ implies $x=0$.
\end{lemma}
\begin{proof}
If we have $\bl(0)\se \bl(x)$, then we must have that $0\in \bl(x)$, and this means that $a\ra x= 0$ for some $a\in L$. But the assignment $a\ra -$ is inflationary, and so $x\leq 0$. 
\end{proof}

\begin{lemma}\label{l:t1P(x)}
    For a spatial frame $\Om(X)$, the pair $(\Om(X),\ca{P}(X))$ is a Raney extension, and $\ca{P}(X)^*=\opp{\fr(\Om(X))}$.
\end{lemma}
\begin{proof}
    Let $X$ be a space, and consider $(\Om(X),\fr(\Om(X)))$. We have that the subcolocale $\ca{P}(X)^*\se \opp{\Fi(\Om(X))}$ is Boolean, by the isomorphism of Theorem \ref{charC*}. We claim that all filters in $\ca{P}(X)$ are regular. This holds because for $S\se X$ we have $\up^{\Om(X)}S=\{U\in \Om(X):S^c\cup U=X\}$, and, indeed $S^c\in \ca{P}(X)$. Thus, we have $\fr(\Om(X))\se \ca{P}(X)^*$. By Lemma \ref{maximalboolean}, then, $\fr(\Om(X))=\ca{P}(X)^*$. 
\end{proof}

\begin{theorem}
The following are equivalent for a $T_0$ space $X$.
\begin{enumerate}
    \item The space $X$ is $T_1$. 
    \item The Raney extension $(\Om(X),\ca{U}(X))$ is $(\Om(X),\ca{P}(X))$.
    \item The Raney extension $(\Om(X),\ca{U}(X))$ is $\ca{R}$-dense.
    \item The Raney extension $(\Om(X),\ca{U}(X))$ is $\ca{R}$-canonical.
    \item The Raney extension $(\Om(X),\ca{U}(X))$ is isomorphic to $(\Om(X),\fr(\Om(X)))$.
    
\end{enumerate}
\begin{proof}
The equivalence between (1) and (2) is a well-known characterization of $T_1$ spaces. If (2) holds, then we have $\ca{R}$-density by Lemma \ref{l:t1P(x)}, and by the characterization in Proposition \ref{C*canonical}. By Corollary \ref{c:all-e-comp}, any Raney extension is $\ca{R}$-compact, thus (3) implies (4). If (4) holds, then (5) follows from the uniqueness result in Theorem \ref{containsprincipal}. If (4) holds, then $\ca{U}(X)^*=\fr(\Om(X))$, and so all neighborhood filters are regular. By Lemma \ref{l:N(x)regular}, (1) follows.
\end{proof}

\end{theorem}
Let us then study the $T_1$ axiom more pointfreely. A topological space is $T_1$ if and only if all subspaces are intersections of open subspaces. This means that a space $X$ is $T_1$ if and only if $\ca{U}(X)$ is the same as the powerset $\ca{P}(X)$. Motivated by this, we define a Raney extension $(L,C)$ to be $T_1$ if and only if $C$ is a Boolean algebra.

\begin{theorem}\label{charsubfit}
    For a frame $L$, the following are equivalent.
    
    \begin{enumerate}
        \item $L$ is subfit.
        \item All exact filters of $L$ are regular.
        \item $(L,\opp{\fe(L)})$ is a $T_1$ Raney extension.
        \item There exists a $T_1$ Raney extension $(L,C)$.
        \item There is a unique $T_1$ Raney extension on $L$, up to isomorphism. This is $(L,\opp{\fr(L)})$.
        \end{enumerate}
\end{theorem}
\begin{proof}
    Suppose that $L$ is a subfit frame. By Proposition \ref{famouschar}, all principal filters are regular filters. By Lemma \ref{principalmin}, this implies that $\fe(L)\se \fr(L)$. Now, suppose that we have $\fe(L)\se \fr(L)$. This implies that $(L,\opp{\fe(L)})=(L,\opp{\fr(L)})$, as regular filters are exact for every frame. By Proposition \ref{reg}, the coframe $\opp{\fr(L)}$ is a Boolean algebra. It is clear that condition (3) implies condition (4). Let us show that (4) implies (5). If $(L,B)$ is a Raney extension such that $B$ is Boolean, as $B^*$ contains all principal filters, we must have $\fe(L)\se B^*$. We also have $\fr(L)\se \fe(L)$, as this holds for all frames, and so by Lemma \ref{maximalboolean} we must have $\fr(L)=\fe(L)=B^*$. Thus, by Corollary \ref{c:=*iso}, $(L,B)$ and $(L,\opp{\fr(L)})$ are isomorphic. Now, suppose that (5) holds. Then, $(L,\opp{\fr(L)})$ is a Raney extension. This means that all principal filters are regular, and so by Proposition \ref{famouschar} the frame $L$ must be subfit.
\end{proof}

For a Raney extension $(L,C)$, we call a \emph{$T_1$ reflection} a map $\tau:(L,C)\to T(L,C)$ such that $T(L,C)$ is $T_1$, and such that whenever $f:(L,C)\to (M,D)$ is a map to a $T_1$ Raney extension, we have a commuting diagram as follows.

\[
\begin{tikzcd}
(L,C)
\ar[r,"\tau"]
\ar[dr,"f",swap]
& T(L,C)
\ar[d,"f_{\delta}"]\\
& (M,D).
\end{tikzcd}
\]

\begin{proposition}
In $\bd{Raney}_{\ca{E}}$, all subfit frames admit a $T_1$ reflection, and this coincides with their $T_D$ reflection. 
\end{proposition}
\begin{proof}
Let $(L,C)$ be a subfit Raney extension. We claim that the reflection map is the map 
\[
\mb{cl}_{\ca{R}}:(L,C)\to (L,\opp{\fr(L)}).
\]
The pair $(L,\opp{\fr(L)})$ is a Raney extension, by Proposition \ref{famouschar}, and it is $T_1$, as $\fr(L)$ is always Boolean. Furthermore, by Corollary \ref{c:all-e-comp}, we have $\opp{\fr(L)}\se C^*$, and so by Theorem \ref{whenliftsfilters} the map above is a map of Raney extensions. The map extends the identity on $L$, and this is, indeed, such that preimages of exact filters are exact. Thus, this is a morphism in $\bd{Raney}_{\ca{E}}$. Now, suppose that there is a $T_1$ Raney extension $(M,D)$ such that we have a morphism $f:(L,C)\to (M,D)$ in $\bd{Raney}_{\ca{E}}$. By Theorem \ref{charsubfit}, $M$ must be subfit. By the same theorem, $D^*=\fr(M)=\fe(M)$. By assumption on $f$, then, the preimage map relative to $f|_{L}$ maps filters in $D^*$ to exact filters of $L$, and by subfitness of $L$ we have $\fe(L)=\fr(L)$. Hence, by Theorem \ref{whenliftsfilters}, we have a map $f_{\tau}:(L,\opp{\fr(L)})\to (M,D)$ extending $f_L$, as required.
\end{proof}

\subsection{Scatteredness}
The notion of scattered space is already present in classical topology, see for example \cite{willard70}. In \cite{simmons80} it is proven that a $T_0$ space is scattered if and only if $\mf{S}(\Om(X))$ is Boolean. This motivates the definition of scattered frame (see \cite{plewe00}): a frame is \emph{scattered} if the coframe $\mf{S}(L)$ is Boolean. As proven in \cite{ball16}, a frame is scattered and subfit if and only if all its sublocales are joins of closed sublocales. This also implies that a subfit frame is scattered if and only if $\mf{S}(L)=\Sc(L)$. Subfit scattered frames are also fit, and so $\mf{S}(L)=\mf{S}_{\op}(L)$, from which we obtain that for $L$ subfit and scattered the equality $\fr(L)=\fe(L)=\fse(L)$ holds (see Theorem \ref{charsubfit}). 

\begin{proposition}
For a subfit frame $L$, the following are equivalent.
\begin{enumerate}
    \item The frame $L$ is scattered.
    \item $\fse(L)=\fe(L)=\fr(L)$.
    \item $\fse(L)=\fe(L)$.
     \item The frame has a unique Raney extension, up to isomorphism.
    \item We have $\mf{S}_{\op}(L)=\mf{S}_{\cl}(L)$.
    \item The frame has a unique Raney extension, up to isomorphism, and this is $(L,\sll)$.
\end{enumerate}
\end{proposition}
\begin{proof}
Suppose that $L$ is a scattered subfit frame. Let $F$ be a strongly exact filter, by Theorem \ref{eandse} we must have that this is $\{x\in L:S\se \op(x)\}$ for some sublocale $S$. By hypothesis, $S$ is a join $\bve_i \cl(x_i)$ of closed sublocales, so that 
\[
F=\{x\in L:x_i\ve x=1\text{ for all }i\in I\}=\bca_i \{x\in L:x\ve x_i=1\}.
\]
By the characterization of regular filters in Proposition \ref{reg}, then, $\fse(L)\se \fr(L)$. This implies (2), as for all frames we have $\fr(L)\se \fe(L)\se \fse(L)$. It is clear that (2) implies (3). Let us show that (3) implies (4). The inclusion $\fe(L)\se \fse(L)$ holds for every frame. Now, suppose that in $L$ every strongly exact filter is exact. For any Raney extension $(L,C)$, we must have $\fe(L)\se C^*\se \fse(L)$. Our assumption, then, implies $\fe(L)=C^*=\fse(L)$. Item (4), then, follows by Corollary \ref{c:=*iso}. Suppose, now, that $L$ has a unique Raney extension, up to isomorphism. The pair $(L,\So(L))$ is a Raney extension. As $L$ is subfit, this must be a Boolean extension, by Theorem \ref{charsubfit}. As $\mf{S}_{\op}(L)$ is a subcoframe of $\mf{S}(L)$, this means that in $\mf{S}(L)$ every fitted sublocale has a complement, which is itself a fitted sublocale. In particular, all joins of closed sublocales are fitted and so $\mf{S}_{\cl}(L)\se\mf{S}_{\op}(L)$. Finally, recall that the lattice $\mf{S}_b(L)$ of joins of complemented sublocales is $\mf{S}_{\cl}(L)$ for subfit frames. We then also have the reverse set inclusion $\mf{S}_{\op}(L)\se \mf{S}_{\cl}(L)$. If (5) holds, by subfitness we have that $(L,\mf{S}_{\op}(L))$ is a Boolean extension. Since this is the largest Raney extension, all its Raney extensions must be Boolean. By Theorem \ref{charsubfit}, when Boolean extensions exist, they are unique. Note also that $\So(L)=\Sc(L)$ implies that every closed sublocale is fitted, and this implies that the frame $L$ is subfit, hence $\So(L)=\mf{S}(L)$. Suppose, finally, that (6) holds. Because all subfit frames have a Boolean extension, by Theorem \ref{charsubfit}, $\sll$ must be Boolean, and so $L$ is scattered.
\end{proof}

\section{Exactness and \texorpdfstring{T\textsubscript{D}}{TD} duality}

We have seen that the spectrum of $(L,\opp{\fe(L)})$ is the $T_D$ spectrum of the frame $L$ (Lemma \ref{spectrumofef}), and that a $T_0$ space is $T_D$ if and only if its Raney extension is $\ca{E}$-canonical (Theorem \ref{t:chartd}). Let us explore more connections between exactness and the $T_D$ axiom.

\begin{lemma}\label{l:tdimpliesexact}
Any $T_D$ frame map $f:L\to M$ such that $M$ is $T_D$-spatial is exact.
\end{lemma}
\begin{proof}
   Suppose that $L$ and $M$ are frames and $M$ is $T_D$-spatial, and that there is a frame map $f:L\to M$ such that $f_*(p)$ is a covered prime whenever $p\in M$ is covered. Now, suppose that $\bwe_i x_i\in L$ is an exact meet. We show $\bwe_i f(x_i)\leq f(\bwe_i x_i)$. Suppose that $p\in M$ is a covered prime with $f(\bwe_i x_i)\leq p$. Then $\bwe_i x_i\leq f_*(p)$, that is, $\bwe_i x_i\ve f_*(p)=f_*(p)$. By exactness, $\bwe_i (x_i\ve f_*(p))=f_*(p)$, and by coveredness there is $i\in I$ with $x_i\ve f_*(p)=f_*(p)$. Then, $f(x_i)\leq p$, which implies $\bwe_i f(x_i)\leq p$, and by $T_D$-spatiality this implies $\bwe_i f(x_i)\leq f(\bwe_i x_i)$ as desired. Let us now show that $\bwe_i f(x_i)$ is exact. Let $y\in M$. Suppose that $\bwe_i f(x_i)\ve y\leq p$ for $p\in M$ a covered prime. As shown above, this means $f(\bwe_i x_i)\ve y\leq p$. Similarly as above, we obtain $f(x_i)\ve y\leq p$ for some $i\in I$, and so $\bwe_i (f(x_i)\ve y)\leq p$. By $T_D$-spatiality, $\bwe_i (f(x_i)\ve y)\leq \bwe_i f(x_i)\ve y$, as desired.
   \end{proof}

\begin{theorem}
    The adjunction $\Om\dashv \pt$ between frames and spaces restricts to an adjunction $\Om:\bd{Top}_D\lra \bd{Frm}_{\ca{E}}^{op}:\pt$, and this restricts to the known $T_D$ duality.
\end{theorem}
\begin{proof}
It suffices to show that the functor $\Om$ maps continuous maps between $T_D$ spaces to exact frame maps, and that the $T_D$ spatialization map of a frame is exact. By Lemma \ref{l:tdimpliesexact}, it is known that the spatialization map is a $T_D$ morphism, so by Lemma \ref{l:tdimpliesexact} it is also exact. By the same Lemma, a map $f:X\to Y$ between $T_D$ spaces determines an exact frame map $\Om(f):\Om(Y)\to \Om(X)$.
\end{proof}

In light of the isomorphism in Theorem \ref{eandse}, Proposition \ref{whenelifts} implies the following.
\begin{corollary}\label{whensclifts}
A frame map $f:L\to M$ is exact if and only if it can be extended to a frame morphism $\Sc(f):\Sc(L)\to \Sc(M)$.
\end{corollary}

We then obtain the following.

\begin{proposition}
    Every map $f:X\to Y$ between $T_D$ spaces lifts to a frame map
    \[
    \Sc(\Om(Y))\to \Sc(\Om(X))
    \]
\end{proposition}
\begin{proof}
    By Lemma \ref{l:tdimpliesexact}, a continuous map $f:X\to Y$ between $T_D$ spaces becomes an exact frame map $\Om(f):\Om(Y)\ra \Om(X)$, and by Corollary \ref{whensclifts} this lifts to a frame map $\Sc(\Om(Y))\to \Sc(\Om(X))$.
\end{proof}

Then, $T_D$ duality remains intact if we replace $\bd{Frm}_D$ with the subcategory $\bd{Frm}_{\ca{E}}$. The advantage of working in this category is that all morphisms $f:L\to M$ lift to morphisms $\Sc(L)\to \Sc(M)$. In fact, we also have the following picture, where the functor $\mf{S}_{\cl}$ is the one mapping a frame $L$ to the Raney extension $(L,\opp{\scl})$.
\begin{center}
    \begin{tikzcd}[row sep=large,column sep=large]
        \bd{Frm}_{\ca{E}}
        \ar[r,"\pt_D"]
        \ar[d,"\Sc"]
        & \bd{Top}_D
        \\
       \bd{Raney}_D,
    \ar[ur,"\rpt"]
    \end{tikzcd}
\end{center}

Here, the category $\bd{Raney}_D$ is the full subcategory of $\bd{Raney}$ determined by the $T_D$ Raney extensions. Note that this is also a full subcategory of $\bd{Raney}_{\ca{E}}$.

\begin{lemma}\label{l:exstability}
For a frame $L$, if a meet $\bwe_i x_i\in L$ is exact, then so is $\bwe_i (x_i\ve y)$ for all $i\in I$.
\end{lemma}
\begin{proof}
    Observe that, if $\bwe_i x_i$ is exact, for all $z\in L$, we have $\bwe_i (x_i \ve y \ve z)\leq (\bwe_i x_i)\ve y\ve z\leq (\bwe_i (x_i\ve y))\ve z$.
\end{proof}
\begin{proposition}
A surjective frame map $f:L\to M$ such that it preserves exact meets is exact.
\end{proposition}
\begin{proof}
Suppose that $\bwe_i x_i$ is exact and that $f:L\to M$ is a frame surjection which preserves exact meets. For $u\in L$, we have $\bwe_i (f(x_i)\ve f(u))=\bwe_i f(x_i\ve u)=f(\bwe_i x_i \ve u)=\bwe_i f(x_i)\ve f(u)$. Since all elements of $M$ are $f(v)$ for some $v\in L$, the meet $\bwe_i f(x_i)$ is exact.
\end{proof}

We say that a sublocale is \emph{exact} if the corresponding surjection is exact. Let us call $\mf{S}_{\ca{E}}(L)$ the ordered collection of exact sublocales of a frame.

\begin{proposition}\label{p:charexactS}
    A sublocale $S$ is exact if and only if for every exact meet $\bwe_i x_i$ and for all $x\in L$ we have $\cl(x_i)\cap S\se \cl(x)$ for all $i\in I$ implies that $\cl(\bwe_i x_i)\cap S\se \cl(x)$.
\end{proposition}
\begin{proof}
The surjection corresponding to a sublocale $S\se L$ is the map $\sigma_S:x\mapsto \bwe\{s\in S:x\leq s\}$. Meets in $\sigma_S[L]=S$ are computed as $\bwe_i^S \sigma_S(x_i)=\bwe \{s\in S:x_i\leq s\mb{ for some $i\in I$}\}$. Exactness of $S$ amounts to having, for every exact meet $\bwe_i x_i$, that $\bwe \{s\in S:x_i\leq s\mb{ for some }i\in I\}\leq \bwe\{s\in S:\bwe_i x_i\leq s\}$. Observe that we can re-write this as $\bwe (\bcu_i S\cap \up x_i)\leq \bwe (S\cap \up\bwe_i x_i)$. By definition of the closure of a sublocale, and by definition of closed sublocale, this means that the condition is also equivalent to $\mi{cl}(S\cap \cl(\bwe_i x_i))\se \mi{cl}(\bve_i (S\cap \cl(x_i)))$, and this is equivalent to the given condition.
\end{proof}
\begin{remark}
    We note that the result above can be generalized: a sublocale $S$ is such that $\sigma_S$ preserves a certain class of meets if and only if for all meets $\bwe_i x_i$ in that class we have, for all $x\in L$, that $\cl(x_i)\cap S\se \cl(x)$ for all $i\in I$ implies that $\cl(\bwe_i x_i)\cap S\se \cl(x)$.
\end{remark}

\begin{proposition}\label{p:ex-suff}
The collection $\Se(L)$ is closed under all joins, and that it contains

\begin{itemize}
    \item All closed sublocales;
    \item All open sublocales;
    \item The two-element sublocales $\bl(p)$ for covered $p$.
\end{itemize}    
\end{proposition}
\begin{proof}
  By Proposition \ref{p:charexactS}, if $S_j$ is a collection of exact sublocales, and $\bwe_i x_i$ an exact meet, then $\cl(x_i)\cap \bve_j S_j\se \cl(x)$ implies that $\cl(x_i)\cap S_j\se \cl(x)$ for all $j$'s, by Lemma \ref{l:linear}. Therefore, for all $j$'s, $\cl(\bwe_i x_i)\cap S_j\se \cl(x)$, and the result follows again by linearity. To see that it contains all closed sublocales, consider that if $\cl(x_i)\cap \cl(y)\se \cl(x)$ then $\cl(x_i\ve y)\se \cl(x)$, that is $x \leq x_i\ve y$, and so $x\leq \bwe_i x_i\ve y$, by exactness, and this is equivalent to $\cl(\bwe_i x_i)\cap \cl(y)\se \cl(x)$. Finally, for open sublocales, we notice that $\cl(x_i)\cap \op(y)\se \cl(x)$ means $\cl(x_i)\se \cl(y)\ve \cl(x)$, and this, by exactness, means $\cl(\bwe_i x_i)\se \cl(y)\ve \cl(x)$, that is $\cl(\bwe_i x_i)\cap \op(y)\se \cl(x)$, as desired. For the third part, consider a covered prime $p\in L$ and suppose that for an exact meet $\bwe_i x_i\in L$ we have $\cl(x_i)\cap \bl(p)\se \cl(x)$. This means that $\bl(p)\se \cl(x)\ve \op(x_i)$ for all $i$'s. Using the properties of prime elements in Lemma \ref{l:prime}, we obtain that either $x\leq p$ or $x_i\nleq p$ for all $i\in I$. In the first case, $\bl(p)\se \cl(x)$, and the desired result follows. In the second case, we have $\bwe_i (x_i\ve p)=\bwe_i x_i\ve p\neq p$, by exactness and coveredness, and so $\bwe_i x_i\nleq p$, from which the desired claim follows.
\end{proof}

\begin{lemma}\label{l:SCLsuff}
If a subcollection $\ca{S}\se \Ss(L)$ is closed under joins and is stable under the operation $-\cap \cl(x)$ and $-\cap \op(x)$ for all $x\in L$, then it is a subcolocale.
\end{lemma}
\begin{proof}
Suppose that $\ca{S}\se \Ss(L)$ is closed under all joins and stable under the two operations above. For it to be a subcolocale, it suffices to show that if $S\in \ca{S}$ and $T\in \Ss(L)$ then $S{\sm}T\in \ca{S}$. Every sublocale of $L$ is of the form $\bca_i \op(x_i)\ve \cl(y_i)$, and $S{\sm}\bca_i \op(x_i)\ve \cl(y_i)=\bve_i (S{\sm}(\op(x_i)\ve \cl(y_i)))$. Then, for $\ca{S}$ to be a subcolocale it suffices for it to be stable under $-{\sm}(\op(x)\ve \cl(y))$. If $\ca{S}$ is as required, and $S\in \ca{S}$, and $x,y\in L$, we have $S\cap \cl(x)\cap \op(y)=S{\sm}(\op(x)\ve \cl(y))\in \ca{S}$.
\end{proof}

\begin{theorem}
    The inclusion $\Se(L)\se \Ss(L)$ is a subcolocale inclusion.
\end{theorem}
\begin{proof}
   By Lemma \ref{l:SCLsuff}, it suffices to show that the collection is closed under all joins and stable under $-\cap \cl(x)$ and $-\cap \op(x)$ for all $x\in L$. The first claim follows from Proposition \ref{p:charexactS}. For the second, suppose that $y\in L$. Suppose that $S$ is exact. We show that $S\cap \op(y)$ is exact. If for exact $\bwe_i x_i$ we have $\cl(x_i)\cap S\cap \op(y)\se \cl(x)$, then $\cl(x_i)\cap S\se \cl(x)\ve \cl(y)=\cl(x\we y)$, and so by hypothesis $\cl(\bwe_i x_i)\cap S\se \cl(x\we y)$, that is $\cl(\bwe_i x_i)\cap S\cap \op(y)\se \cl(x)$. Let us show that $S\cap \cl(y)$ is exact. If for exact $\bwe_i x_i$ we have $\cl(x_i)\cap S\cap \cl(y)\se \cl(x)$ then $\cl(x_i\ve y)\cap S\se \cl(x)$, and since $\bwe_i (x_i\ve y)$ is exact by Lemma \ref{l:exstability}, and by exactness of $\bwe_i x_i$, this implies that $\cl(\bwe_i x_i\ve y)\cap S\se \cl(x)$, that is $\cl(\bwe_i x_i)\cap S\cap \cl(y)\se \cl(x)$.
\end{proof}

For every frame $L$, we have subcolocale inclusions $\Se(L)\se \mf{S}_{D}(L)\se \mf{S}(L)$. We do not know, yet, how to characterize frames for which $\Se(L)=\mf{S}_D(L)$, or those such that $\Se(L)=\mf{S}(L)$, and leave this as an open question.

\begin{corollary}
    A spatial sublocale is exact if and only if it is a D-sublocale.
\end{corollary}
\begin{proof}
   This follows from Lemma \ref{l:tdimpliesexact}.
\end{proof}

 \printbibliography
\end{document}